\documentclass[conference]{IEEEtran}
\usepackage{cite}

\ifCLASSINFOpdf
  \usepackage[pdftex]{graphicx}
\else
\fi
\usepackage{bm}
\usepackage{amsthm}
\usepackage{amsmath}
\usepackage{amsfonts}
\newcommand{\Z}{\mathbb Z}

\newcommand{\B}{\mathrm{B}}

\newtheorem*{Prop}{Proposition}

\usepackage{url}

\begin{document}
%
\title{Constructing Large-scale Low-latency Network from Small Optimal Networks}

\author{\IEEEauthorblockN{Ryosuke Mizuno}
\IEEEauthorblockA{Department of Physics\\
 Kyoto University, Kyoto, Japan\\
Email: mizuno@tap.scphys.kyoto-u.ac.jp}
\and
\IEEEauthorblockN{Yawara Ishida}
\IEEEauthorblockA{Research Institute for Mathematical Science\\
Kyoto University, Kyoto, Japan\\
Email: yawara@kurims.kyoto-u.ac.jp}
}


%


\maketitle

\begin{abstract}
The construction of large-scale, low-latency networks becomes difficult as the number of nodes increases.
In general, the way to construct a theoretically optimal solution is unknown.
However, it is known that some methods can construct suboptimal networks with low-latency.
One such method is to construct large-scale networks from optimal or suboptimal small networks, using the product of graphs.
There are two major advantages to this method.
One is that we can reuse small, already known networks to construct large-scale networks.
The other is that the networks obtained by this method have graph-theoretical symmetry, 
which reduces the overhead of communication between nodes.
A network can be viewed as a ``graph'', which is a mathematical term from combinatorics.
The design of low-latency networks can be treated as a mathematical problem 
of finding small diameter graphs with a given number of nodes ( called ``order'' )
and a given number of connections between each node ( called ``degree'' ).
In this paper, we overview how to construct large graphs from optimal or suboptimal small graphs by using graph-theoretical products.
We focus on the case of diameter 2 in particular.
As an example, we introduce a graph of order 256, degree 22 and diameter 2,
which granted us the Deepest Improvement Award at the ``Graph Golf'' competition.
Moreover, the average shortest path length of the graph is the smallest in graphs of order 256 and degree 22. 
\end{abstract}



%
\IEEEpeerreviewmaketitle

\section{Introduction}
In mathematics, graph theory is the study of graphs, which are mathematical objects
obtained by abstraction of networks.
A {\it graph} $G$ is a network which consists of a set of {\it vertices} $V$ and of {\it edges} $E$.
Nodes and connections between nodes are translated to vertices and edges respectively.
Two distinct vertices $u$ and $v$ of $G$ are connected to each other
if and only if $(u,v) \in E$. 
In this case, we say that $u$ and $v$ are {\it adjacent} and use the notation $u \sim v$.
For example, let $V$ be a set of $(i,j)$ such that $i = 0,1$ and $j=0,1, \ldots ,4$, and let $E$ be a set of $((i,j),(k,l))$
such that if $i=k$ then $j=l \pm 1 \bmod 5$ or if $i \neq k$ then $j = 2l \bmod 5$.
The graph of the given $V$ and $E$ is called {\it Petersen graph} as shown in Figure~\ref{fig:Petersen}.
\begin{figure}[]
\centering
\includegraphics[width=3cm]{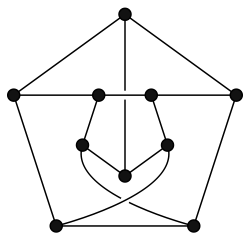}
\caption{ Petersen graph.}
\label{fig:Petersen}
\end{figure}
We can study networks by modeling them as graphs and formulate various problems in terms of graph theory.
For example, the traveling salesman problem, the four color theorem, the max-flow min-cut theorem
and the shortest path problem are well known formulations of problems of networks.
Knowledge gained by the formulation of graph theory is expected to be useful for the study of networks.

The {\it order} $N=|G|=|V|$ and the {\it maximum degree} of vertices $\Delta$ are basic feature values of
graphs, where the degree of a vertex is the number of adjacent vertices to it.
The {\it distance} $d(u,v)$ between $u$ and $v$ in $V$  is the length of the shortest path between them.
The {\it diameter} $D$ of the graph $G$ is the maximum distance of all pairs of vertices;
\[ D = \max_{u,  v \in V} d(u,v). \]
The diameter of Petersen graph is $2$.
Let us turn to some other famous examples of graphs.
Let $V$ be vertices of the $n$-dimensional hypercube over a field such that each component is $1$ or $0$,
and the vertices ${\bm x}$ and ${\bm y}$ in $V$ are adjacent if and only if $1$ or $-1$ appears just once in the components of ${\bm x}-{\bm y}$, and other components of it are $0$.
This graph is called an {\it $n$-dimensional hypercube graph},
which is of $N=2^n, \Delta=n, D=n$.
Let $V$ be ${(\Z/m\Z)}^n$ such that $m > 2$, and two vertices ${\bm x}$ and ${\bm y}$ in $V$ are adjacent
if and only if $1$ or $m-1$ appears just once in the components of ${\bm x}-{\bm y}$, and other components of it are $0$.
This graph is called an {\it $n$-dimensional torus grid graph},
which is of $N=m^n, \Delta=2n, D=n [m/2]$ where this bracket is gauss's symbol.
These two examples are graphs obtained by abstracting known topological structures.
On the other hand, we can also construct graphs directly by using algebraic methods.
Let $L$ be a set of distinct $t$ labels and let the vertices $V$ be $L^n$.
The vertices $(a_1, a_2, \ldots, a_n)$ and $(b_1, b_2, \ldots, b_n)$ in $V$ are adjacent if and only if $a_i = b_{i+1}$ or $a_{i+1} = b_i$ for $i=1,2, \ldots k-1$.
This graph is known as undirected {\it de Brujin graph} of type $(t,n)$
of $N=t^n, \Delta=2t, D=n$~\cite{bruijn1946combinatorial,good1946normal}.

The order, the maximum degree, and the diameter are important indices for measuring
the efficiency of connections between vertices of a graph.
If a graph of given diameter and maximum degree has more vertices, we consider it as a more efficient network.
This claim can be rephrased into two other forms.
One is that a graph of given order and maximum degree is more dense if its diameter is smaller.
The other is that if the maximum degree of a graph of given order and diameter is smaller,
the distribution of the edges is more efficient.
These claims correspond to three optimization problems; {\it the degree-diameter problem},
{\it the order-degree problem} and {\it the order-diameter problem}.
The degree-diameter problem is to find the largest possible number of vertices in a graph of given degree and diameter~\cite{MilSir2005}.
Of the above three, this is the most attractive to mathematicians.
A table of a lower bound of the largest possible number of vertices with given degree and diameter is available online at the Combinatorics Wiki website \cite{combinatoricswiki}.
The order-degree problem is to find graphs with the smallest diameter in graphs of given order and degree.
This problem can be applied for designing of law-latency networks, but a table similar to the degree-diameter problem has not been constructed as of yet. 
Only a very small table is available online at the Graph Golf website \cite{graphgolf}.
The order-degree problem is to find graphs with the smallest maximum degree in a graph of given order and diameter.
Unfortunately, this problem has been somewhat overlooked compared to the others.
These three problems are strongly interrelated.
Therefore we can use the best-studied degree-diameter problem in order to study the others.
In what follows, we review the degree-diameter problem and two methods for constructing graphs of diameter $2$ for the order-degree problem.
One is {\it Brown's construction}, which was originally introduced by W.G. Brown ~\cite{brown1966graphs}
and later generalized by the co-author\cite{gbc} of this paper.
The other is graph-theoretical product known as {\it star product}\cite{bermond1982large,MilSir2005}.
This allows us to introduce the graph of order $256$, degree $22$, and diameter $2$.
Moreover, the average shortest path length of the graph is the smallest in graphs of order $256$ and degree $22$,
where the average shortest path length $ASPL$ is defined as follows;
\[ ASPL = \frac{\sum_{u, v \in V} d(u, v)}{|V| (|V| -1)} \]

\section{The degree-diameter problem}
As mentioned above, the degree-diameter problem is to find
the largest possible number  $n_{\Delta,D}$ of vertices in a graph of given degree $\Delta$ and diameter $D$~\cite{MilSir2005}.
If $G$ is a graph with degree $\Delta$ and diameter $D$, then we get
\[ |G| \leq n_{\Delta,D} \leq 1 + \Delta \sum_{k=0}^{D-1} (\Delta - 1)^k\]
where $|G|$ is the number of vertices of $G$.
The right hand side of the above inequality is called {\it Moore bound}.
The graph whose order equals Moore bound is called {\it Moore graph}.
In the case of $D=2$, Moore graph can exists when $\Delta=2, 3, 7, 57$~\cite{hoffman1960moore}.
In fact, if $\Delta=2, 3, 7$, then Moore graphs are a pentagon, Petersen graph, Hoffman-Singleton graph respectively.
It is not known whether Moore graph exists if $\Delta=57$.
In the case of $D \geq 3$, Moore graph is a cycle graph of length $2D+1$~\cite{damerell1973moore}.
Some general lower bounds of $n_{\Delta, D}$ are known( see ~\cite{MilSir2005} ).
For example, if $\Delta$ is an even number, the undirected de Brujin graph gives a lower bound of $n_{\Delta, D}$;
\[ n_{\Delta, D} \geq {\left( \frac{\Delta}{2} \right)}^D \]

However, exact lower bounds of $n_{\Delta, D}$ are only known for small $\Delta, D$.
Thus, the order-diameter problem for almost all pairs of $\Delta, D$ is unsolved.
Let us consider the case of $\Delta=8, D=8$, where Moore bound is $7686401$.
The $8$-dimensional hypercube graph, the $4$-dimensional torus grid graph over $\Z/5\Z$,
and the undirect de Brujin graph of type $(4,8)$ are examples of $\Delta=8,D=8$.
The order of the $8$-dimensional hypercube graph is $2^8=256$.
The order of the $4$-dimensional torus grid graph over $\Z/5\Z$ is $5^4=625$, 
which is more efficient than $8$-dimensional hypercube.
The order of undirect de Brujin graph of type $(4,8)$ is $4^8=65536$,
which is more efficient than the above two graphs.
The order of the already known suboptimal graph of $\Delta=8, D=8$ is $734820$~\cite{Loz08newrecord}.
The graph is discovered by using a computer.
However it seems to be too small because the percentage of the Moore bound for $\Delta=8, D=8$ is $9.56$\%.
As another example, let us consider the case of $\Delta=20, D=2$.
The order the undirected de Brujin graph of type $(10, 2)$ of $\Delta=20, D=2$ is $10^2=100$.
The order of the already known suboptimal graph  of $\Delta=20, D=2$ is $381$,
where the percentage of the Moore bound for $\Delta=20, D=2$ is $95$\%.
It is obtained by Brown's Construction, which we will discuss in more detail later on.
As mentioned above, the hypercube graph, the torus grid graph, and the de Brujin graph are very simple and easy to construct,
but these graphs are not close to suboptimal.
Thus, when we want more dense or more efficient graphs, these simple graphs are not suitable.
In order to construct more dense or more efficient graphs, there are roughly two ways;
either by a deterministic mathematical construction, or by searching suboptimal graphs using computers.
Furthermore, a deterministic mathematical construction can be divided into two ways.
One way is to construct graphs with more vertices from scratch.
The other is to construct graphs from small, already known suboptimal graphs.

\subsection{The case of diameter $2$}
For the case of diameter $2$, Brown's construction~\cite{brown1966graphs,MilSir2005} gives suboptimal graphs.
This construction gives a graph for each $q$ which is a power of a prime.
Let $F_q$ be a finite field.
Brown's construction gives the graph $\B(F_q)$ where the vertices are lines in $F_q^3$ and two lines are adjacent if and only if they are orthogonal.
It follows that;
\[ |\B(F_q)| = q^2+q+1, \quad \Delta(\B(F_q)) = q+1, \quad D(\B(F_q))=2.\]
The degree of each vertex of $\B(F_q)$ is $q+1$ or $q$.
The reason of $D(\B(F_q))=2$ is that
for all two vectors ${\bm x}$ and ${\bm y}$ there exists a vector ${\bm z}$ orthogonal to ${\bm x}$ and ${\bm y}$ even in finite vector spaces.
Therefore, for ${\Delta}$, if a power of a prime $q$ exists such that $\Delta=q+1$, we get
\[ n_{\Delta, D} \geq \Delta^2 - \Delta + 1. \]
Among $q^2+q+1$ vertices, $q+1$ vertices are of degree $q$ and $q^2$ vertices are of degree $q+1$.
If $q$ is a power of $2$, there exists the graph of $N=q^2+q+2, \Delta=q+1,D=2$~\cite{journals/networks/ErdosFH80}.
Therefore, if $\Delta$ is a power of 2, we have
\[ n_{\Delta, D} \geq \Delta^2 - \Delta + 2. \]
For the remainder $\Delta$, we get suboptimal graphs by duplicating the vertex of the graphs obtained by Brown's construction~\cite{MilSir2005}.
For any $\epsilon >0$ there exists a constant $c_\epsilon$ such that, for any $\Delta$ the following holds;
\[ n_{\Delta, D} \geq \Delta^2 - c_\epsilon \Delta^{19/12+\epsilon}. \]
For the case of diameter $2$ with large maximum degree, the graphs obtained by Brown's construction
are the best suboptimal graphs among already known graphs.

Let us turn to another graph-theoretical technique known as star product, which was introduced by Bermond, Delorme and Farhi\cite{bermond1982large,MilSir2005}.
Let $G_1, G_2$ be graphs.
We fix an arbitrary orientation of all edges of $G_1$ and let $\vec{E}$ be the corresponding set of the fixed arrows of $G_1$.
For each arrow $(u,v) \in \vec{E}$, let $\phi(u,v)$ be a bijection on the set $V(G_2)$.
The vertex set of the star product $G_1 *_\phi G_2$ is thus $V(G_1) \times V(G_2)$,
and the vertex $(u, v)$ is adjacent to $(w, x)$ if and only if either $u = w$ and $(v,x)$ is an edge of $G_2$, or
$(u, w)$ is in $\vec{E}$ and $x=\phi(u,w)(v)$.
Using this product, we can construct some efficient graphs.
For example, we can construct the graph that gives the exact lower bound of order for $\Delta=6, D=2$.
The exact lower bound for $\Delta=6, D=2$ is 32, namely $n_{6,2} = 32$.
Let the graph $G_8$ of order $8$ be as follows;
\[ V=\{(0,0),(0,1),(1,0),(1,1),(1,2),(2,0),(2,1),(2,2)\} \]
and $(i,j), (k,l)$ in $V$ be adjacent if and only if
$(i,k)$ in $\{ (0,1),(1,0),(2,2)\}$, or if $(i,k)=(1,2),(2,1)$ then $j=l$.
$G_8$ is of order $8$, degree $8$, and diameter $2$ as shown in Figure~\ref{fig:g8}.
\begin{figure}[]
\centering
\includegraphics[width=5cm]{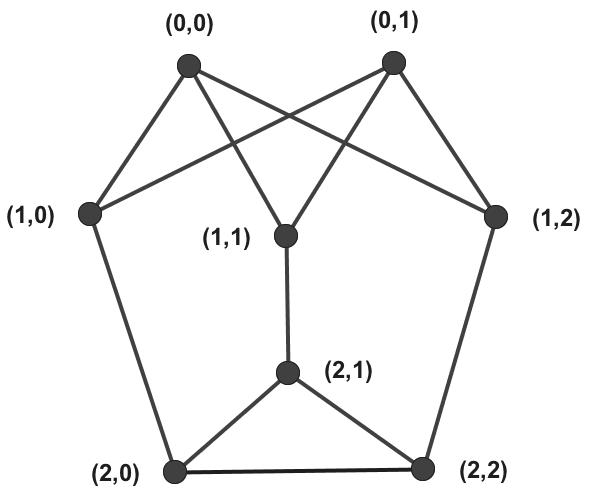}
\caption{ The $G_8$ Construction.}
\label{fig:g8}
\end{figure}
Let $K_n$ be a $n$-complete graph, where each vertex is connected to all other vertices.
The vertices of $K_n$ is labeled by natural numbers, namely $V(K_n)=\{0,1,\ldots, n-1\}$.
Let $\vec{E}$ be $\{(k_1,k_2)| k_1 < k_2\}$ and $\phi$ be defined as follows;
\[
  \phi(k_1,k_2)((i,j)) =
    \begin{cases}
      (0,1-j) & (i=0) \\
      (2,j+1 \bmod 3) & (i=1) \\
      (1,j-1 \bmod 3) & (i=2).
    \end{cases}
\]
Then we get the star product $K_n *_\phi G_8$ with $\phi$, which is of order $8n$, degree $n+2$ such that if $n \geq 3$, the diameter is $2$.
When $n=4$, the $K_4 *_\phi G_8$ is the optimal graph of order $32$ for the degree-diameter problem of $\Delta=6, D=2$.

\section{The order-degree problem}
The order-degree problem is to find the graph with the smallest diameter in graphs of given order and degree.
As mentioned above, the order-degree problem is correlated strongly with the degree-diameter problem.
In fact, for given degree $\Delta$ and diameter $D$, the optimal graph $G$ for the degree-diameter problem is also
optimal for the order-degree problem of order $|G|$ and degree $\Delta$.
If $G$ is not optimal for the order-degree problem, then there exists the graph $G'$ of order $|G|$, degree $\Delta$ and diameter $D' < D$.
Thus, we get $G''$ of order $|G|+1$, degree $\Delta$ and $D'+1 \leq D$ by inserting one vertex into an arbitrary edge of $G'$.
This contradicts the assumption that $G$ is the largest graph of given degree $\Delta$ and diameter $D$.
Therefore, the already known suboptimal graphs  for degree-diameter problem seem to be suboptimal for order-degree problem.
However, the order of optimal or suboptimal graphs of the degree-diameter problem does not cover an arbitrary order.
For example, the facts of $n_{6,2}=32$ and $n_{7,2}=50$ implies nothing more than
 that if $33 \leq |G| \leq 49$ and $D=2$, then $\Delta \geq 7$, or if $33 \leq |G| \leq 49$ and $\Delta=6$, then $D \geq 3$.
Namely, the graphs obtained from the degree-diameter problem cover a very limited number of pairs of order and degree for the order-degree problem.
Thus, we have to construct graphs for missing pairs of given order and degree.
As well as the degree-diameter problem, there are roughly two deterministic ways to construct graphs;
to construct graphs from scratch or from already known small graphs.
We introduce two ways corresponding to them; generalized Brown's construction and multiple star product,
which are the generalization of Brown's construction and star product respectively.
Using multiple star product, we construct the graph of order $256$, degree $22$, and diameter $2$.

\subsection{Generalized Brown's Construction}
We give a generalization of Brown's construction~\cite{gbc} by replacing a finite field $F_q$ with a finite commutative ring $R$ with unity, in particular $\Z/n\Z$.
The vertices $V(\B(R))$ is defined as follows;
\[( R^3 \setminus \{\bm v | \exists r \in R, r \cdot {\bm v} = {\bm 0} \} ) / \sim\]
where $\bm v \sim \bm w$ if and only if there exists $k \in R^*$ such that $k \cdot {\bm v} = {\bm w}$.
The two vertices $[\bm v]$ and $[\bm w]$ are adjacent if and only if ${\bm v} \cdot {\bm w} = 0$.
It is clear that, if $R$ is finite, then so is $\B(R)$.
Moreover if $R$ is a quotient of any Euclidean domain, then the diameter of $\B(R)$ is $2$.
When $R=\Z/n\Z$, there exist prime numbers $p_i$ and natural numbers $k_i > 0$ such that $n=\prod_{i} {p_i}^{k_i}$,
and it follows that;
\[ |\B(\Z/n\Z)| = \prod_i \left( {p_i}^{2k_i}+{p_i}^{2k_i-1}+{p_i}^{2k_i-2} \right)\]
\[ \Delta(\B(\Z/n\Z)) = \prod_i \left( {p_i}^{k_i} + {p_i}^{k_i-1} \right). \]
\[ D(\B(\Z/n\Z)) = 2\]
Therefore, for given order $N$ and degree $\Delta$, if there exists a natural number $n \geq 2$ and $\delta \geq 0$
such that $N=|\B(\Z/n\Z)|+\delta$ and $\Delta \geq \Delta(\B(\Z/n\Z))+\delta$, 
then we have the graph of order $N$, degree $\Delta$ and diameter $2$
by $\delta$ times duplicating vertices of $\B(\Z/n\Z)$ and adding some edges.
We omit the proof of the above discussion because the details will appear in our next paper\cite{nextpaper}.

\subsection{Multiple Star Product}
Let us construct the graph of order $256$, degree $22$, and diameter $2$.
It was necessary to develop a novel method for the construction of this graph because it could not be obtained by generalized Brown's construction nor by the ordinal star product.
For example, the order of $K_{32} *_\phi G_8$ is $256=32 \times 8$ with an arbitrary $\phi$ but the degree is $34=31+3$.
We introduce the new concept of $m$-multiple star product.
Let $G_1, G_2$ be graphs.
We fix an arbitrary orientation of all edges of $G_1$ and let $\vec{E}$ be the corresponding set of the fixed arrows of $G_1$.
For each arrow $(u,v) \in \vec{E}$, let $\psi(u,v,l)$ be a bijection on the set $V(G_2)$ where $l=0,1,\ldots,m-1$.
Then the vertex set of the multiple star product $G_1 *_\psi G_2$ with $\psi$ is $V(G_1) \times V(G_2)$,
and the vertex $(u, v)$ is adjacent to $(w, x)$ if and only if either $u = w$ and $(v,x)$ is an edge of $G_2$, or
$(u, w)$ is in $\vec{E}$ and there exists $l$ such that $x=\psi(u,w,l)(v)$.

We take $K_a, K_b *_\phi G_8$ as $G_1, G_2$ respectively, according to the definition of $\phi$ as seen in the previous section.
Let $\vec{E}$ be $\{(l_1,l_2)| l_1 < l_2\}$ and $\psi$ be defined as follows;
\[
  \psi(l_1,l_2,0)((k,(i,j))) =
    \begin{cases}
      (k,(0,0)) & ((i,j)=(0,0))\\
      (k,(0,1)) & ((i,j)=(1,0))\\
      (k,(1,0)) & ((i,j)=(1,1))\\
      (k,(2,1)) & ((i,j)=(1,2))\\

      (k,(1,1)) & ((i,j)=(0,1))\\
      (k,(2,2)) & ((i,j)=(2,0))\\
      (k,(1,2)) & ((i,j)=(2,1))\\
      (k,(2,0)) & ((i,j)=(2,2))\\
    \end{cases}
\]
\[
  \psi(l_1,l_2,1)((k,(i,j))) =
    \begin{cases}
      (k,(0,1)) & ((i,j)=(0,0))\\
      (k,(1,0)) & ((i,j)=(1,0))\\
      (k,(2,1)) & ((i,j)=(1,1))\\
      (k,(0,0)) & ((i,j)=(1,2))\\

      (k,(2,2)) & ((i,j)=(0,1))\\
      (k,(1,2)) & ((i,j)=(2,0))\\
      (k,(2,0)) & ((i,j)=(2,1))\\
      (k,(1,1)) & ((i,j)=(2,2))\\
    \end{cases}
\]
\[
  \psi(l_1,l_2,2)((k,(i,j))) =
    \begin{cases}
      (k,(1,0)) & ((i,j)=(0,0))\\
      (k,(2,1)) & ((i,j)=(1,0))\\
      (k,(0,0)) & ((i,j)=(1,1))\\
      (k,(0,1)) & ((i,j)=(1,2))\\

      (k,(1,2)) & ((i,j)=(0,1))\\
      (k,(2,0)) & ((i,j)=(2,0))\\
      (k,(1,1)) & ((i,j)=(2,1))\\
      (k,(2,2)) & ((i,j)=(2,2))\\
    \end{cases}
\]
\[
  \psi(l_1,l_2,3)((k,(i,j))) =
    \begin{cases}
      (k,(2,1)) & ((i,j)=(0,0))\\
      (k,(0,0)) & ((i,j)=(1,0))\\
      (k,(0,1)) & ((i,j)=(1,1))\\
      (k,(1,0)) & ((i,j)=(1,2))\\

      (k,(2,0)) & ((i,j)=(0,1))\\
      (k,(1,1)) & ((i,j)=(2,0))\\
      (k,(2,2)) & ((i,j)=(2,1))\\
      (k,(1,2)) & ((i,j)=(2,2))\\
    \end{cases}
\]
Thus, we have the $4$-multiple star product $K_a *_\psi (K_b *_\phi G_8)$.
It can be easily shown that the order and the maximum degree of the graph is $8ab$ and $4a+b-2$ respectively.
We have to prove the remaining property that the diameter of the graph is $2$.
\begin{Prop}
The following equation holds.
\[ D( K_a *_\psi (K_b *_\phi G_8) ) = 2 \]
\end{Prop}
\begin{proof}
The strategy of this proof is straightforward;
traveling from one arbitrary vertex to any other takes two steps at most.
We show that, for all $l_1, l_2$ in $V(K_a)$, $k_1, k_2$ in $V(K_b)$, $u, v$ in $V(G_8)$,
$(l_1,(k_1,u)) \sim (l_2,(k_2, v))$ holds or there exists $w$ in $V(K_a *_\psi (K_b *_\phi G_8))$ such that $(l_1,(k_1,u)) \sim w \sim (l_2,(k_2, v))$.
It is sufficient to prove that $l_1 \neq l_2$ because the diameter of $K_b *_\phi G_8$ is $2$.
Let us define $A, B, C, D \subset V(G_8)$ as follows;
\[A=\{(0,0),(1,0),(1,1),(1,2)\}\]
\[B=\{(0,1),(2,0),(2,1),(2,2)\}\]
\[C=\{(0,0),(0,1),(1,0),(2,1)\}\]
\[D=\{(1,1),(2,2),(1,2),(2,0)\}\]
Thus, we get $V(G_8) = A \sqcup B = C \sqcup D$.
Additionally, let $G$ be a graph and $S, T \subset V(G)$, we use the notation $S \sim T$
if and only if for all $s$ in $S$ there exists $t$ in $T$ such that $s \neq t$, $s\sim t$, and
for all $t$ in $T$ there exists $s$ in $S$ such that $t \neq s$, $t \sim s$.
Using this adjacency of subsets of vertices, for example that $C \sim D$ holds in $G_8$.
We use the further notation $(l, k, S) = \{ (l, k, s) | s \in S\}$.
For $k_1 \neq k_2$, the following holds;
\[ (l_1,k_1,A) \sim_\psi (l_2,k_1,C) \sim_\phi (l_2,k_2,C) \]
\[ (l_1,k_1,A) \sim_\phi (l_1,k_2,B) \sim_\psi (l_2,k_2,D) \]
\[ (l_1,k_1,B) \sim_\psi (l_2,k_1,D) \sim_\phi (l_2,k_2,D) \]
\[ (l_1,k_1,B) \sim_\phi (l_1,k_2,A) \sim_\psi (l_2,k_2,C) \]
where $\sim_\psi$ and $\sim_\phi$ are the adjacencies that results from by $\psi$ and $\phi$ respectively.
For example, $(l_1,k_1,A) \sim_\psi (l_2,k_1,C)$ means that
for all $a$ in $A$ there exists $c$ in $C$ and a natural number $n \leq 3$, $\psi(l_1, l_2, n)((k_1,a)=(k_1,c)$,
and for all $c$ in $C$ there exists $a$ in $A$ and a natural number $n \leq 3$, $\psi(l_1, l_2, n)((k_1,a)=(k_1,c)$.
Therefore, there exists $W \subset V(K_a *_\psi (K_b *_\phi G_8))$
such that $(l_1,k_1,V(G_8)) \sim W \sim (l_2,k_2,V(G_8))$.
Namely, if $k_1 \neq k_2$, for all $u, v$ in $(V(G_8))$, traveling from $(l_1,k_1,u)$ to $(l_2,k_2,v)$ takes just two steps.
For $k_1 = k_2 = k$, the following holds;
\[ (l_1,k,A) \sim_\psi (l_2,k,C) \sim_{G_8} (l_2,k,D) \]
\[ (l_1,k,B) \sim_\psi (l_2,k,D) \sim_{G_8} (l_2,k,C) \]
where $(l_2,k,C) \sim_{G_8} (l_2,k,D)$ means that $C \sim D$ in $G_8$.
The above adjacency implies that
for all $u, v$ in $V(G_8)$ traveling from $(l_1,k,u)$ to $(l_2,k,v)$ takes two steps at most.
\end{proof}
It is clear that the diameters of these graphs $K_a *_\psi (K_b *_\phi G_8)$ are the smallest in graphs of order $8ab$ and degree $4a+b-2$.
Moreover, the average shortest path lengths of the graphs are the smallest in graphs of order $8ab$ and degree $4a+b-2$ because the graphs are regular graphs and the diameters of these are $2$,
where ``regular'' means that all degree of a vertex is the same.
The percentage of the order of the Moore bound is as follows;
\[ \frac{8ab}{(4a+b-2)^2+1} \geq \frac{1}{1+\frac{16a^2+b^2}{8ab}} \]
If $4a=b$, then the percentage is greater than $50$\%. 

When $a=4$ and $b=8$, we get $K_4 *_\psi (K_8 *_\phi G_8)$, which is the graph of order $256$, degree $22$ and diameter $2$.

\section{Conclusion}
We overviewed the degree-diameter problem, which has been best-studied,
and showed two deterministic ways to construct graphs of diameter $2$ for the order-degree problem.
In the first way, we can construct low hop-count graphs for infinite number of pairs of order and degree, with graph-theoretic symmetry.
In the second way, we can automatically obtain large-scale graphs from smaller graphs. 
The order and degree of graphs obtained by these methods are limited to a set of values that satisfy specific relations.
Searching for appropriate graphs using the computer has some merit because it can construct graphs for all pairs of order and degree.
This method is suitable for large diameter graphs because deterministic methods for such cases are not well known.
But computer-based research has a disadvantage that, as the vertices of graphs increase,
searching and calculation of graphs become harder and less time efficient.
Therefore, these two methods for the construction of graphs have to be studied and developed simultaneously.
Similarly to the degree-diameter problem, it is necessary to create the table of greater bounds for the order-degree problem,
which has been somewhat overlooked in the field.
We are planning to make such a table available online in the near future.

\section*{Acknowledgment}
We would like to thank Miikael-Aadam Lotman and Masanori Kanazu.



\bibliographystyle{IEEEtran}
\bibliography{IEEEabrv,ref}

\begin{thebibliography}{10}
\providecommand{\url}[1]{#1}
\csname url@samestyle\endcsname
\providecommand{\newblock}{\relax}
\providecommand{\bibinfo}[2]{#2}
\providecommand{\BIBentrySTDinterwordspacing}{\spaceskip=0pt\relax}
\providecommand{\BIBentryALTinterwordstretchfactor}{4}
\providecommand{\BIBentryALTinterwordspacing}{\spaceskip=\fontdimen2\font plus
\BIBentryALTinterwordstretchfactor\fontdimen3\font minus
  \fontdimen4\font\relax}
\providecommand{\BIBforeignlanguage}[2]{{%
\expandafter\ifx\csname l@#1\endcsname\relax
\typeout{** WARNING: IEEEtran.bst: No hyphenation pattern has been}%
\typeout{** loaded for the language `#1'. Using the pattern for}%
\typeout{** the default language instead.}%
\else
\language=\csname l@#1\endcsname
\fi
#2}}
\providecommand{\BIBdecl}{\relax}
\BIBdecl

\bibitem{bruijn1946combinatorial}
d.~N. Bruijn, ``A combinatorial problem,'' \emph{Proceedings of the Koninklijke
  Nederlandse Akademie van Wetenschappen. Series A}, vol.~49, no.~7, p. 758,
  1946.

\bibitem{good1946normal}
I.~J. Good, ``Normal recurring decimals,'' \emph{Journal of the London
  Mathematical Society}, vol.~1, no.~3, pp. 167--169, 1946.

\bibitem{MilSir2005}
M.~Miller and J.~\v{S}ir\'a\v{n}, ``Moore graphs and beyond: a survey of the
  degree/diameter problem,'' \emph{The Electronic Journal of Combinatorics},
  no. DS14, 2005.

\bibitem{combinatoricswiki}
``The degree diameter problem for general graphs,''
  \url{http://combinatoricswiki.org/wiki/The_Degree_Diameter_Problem_for_General_Graphs}.

\bibitem{graphgolf}
``Graph golf, the order/degree problem competition,''
  \url{http://research.nii.ac.jp/graphgolf/}.

\bibitem{brown1966graphs}
W.~G. Brown, ``On graphs that do not contain a thomsen graph,'' \emph{Canad.
  Math. Bull}, vol.~9, no.~2, pp. 1--2, 1966.

\bibitem{gbc}
Y.~Ishida, ``A generalization of brown's construction for the degree/diameter
  problem,'' \emph{arXiv preprint arXiv:1512.08961}, 2015.

\bibitem{bermond1982large}
J.-C. Bermond, C.~Delorme, and G.~Farhi, ``Large graphs with given degree and
  diameter iii,'' \emph{North-Holland Mathematics Studies}, vol.~62, pp.
  23--31, 1982.

\bibitem{hoffman1960moore}
A.~J. Hoffman and R.~R. Singleton, ``On moore graphs with diameters 2 and 3,''
  \emph{IBM Journal of Research and Development}, vol.~4, no.~5, pp. 497--504,
  1960.

\bibitem{damerell1973moore}
R.~Damerell, ``on moore graphs,'' in \emph{Mathematical Proceedings of the
  Cambridge Philosophical Society}, vol.~74, no.~02.\hskip 1em plus 0.5em minus
  0.4em\relax Cambridge Univ Press, 1973, pp. 227--236.

\bibitem{Loz08newrecord}
E.~Loz and J.~Siráň, ``New record graphs in the degree-diameter problem,''
  2008.

\bibitem{journals/networks/ErdosFH80}
\BIBentryALTinterwordspacing
P.~Erd{\"o}s, S.~Fajtlowicz, and A.~J. Hoffman, ``Maximum degree in graphs of
  diameter 2,'' \emph{Networks}, vol.~10, no.~1, pp. 87--90, 1980. [Online].
  Available: \url{http://dx.doi.org/10.1002/net.3230100109}
\BIBentrySTDinterwordspacing

\bibitem{nextpaper}
Y.~Ishida and R.~Mizuno, ``A graph of diameter 2 for the order-degree
  problem,'' in prepared.

\end{thebibliography}
%
\end{document}